\documentclass[letter]{article}  
\usepackage[width=14.6cm]{geometry}

\usepackage{amsmath, amsfonts, amsthm, amssymb, xypic, color, epigraph, mathrsfs, calc}
\usepackage[hidelinks]{hyperref}

\usepackage{tensor,dsfont, enumerate, enumitem}
\usepackage{tikz}
\usetikzlibrary{arrows, automata}

\usepackage{graphicx}
\usepackage{float}
\usepackage{verbatim}
\usepackage{makeidx}
\usepackage{longtable}

\numberwithin{equation}{section}

\newcommand{\ie}{\emph{i.e.\ }}
\newcommand{\eg}{\emph{e.g.\ }}

\newtheorem{theo}{Theorem}[section]
\newtheorem{lma}[theo]{Lemma}
\newtheorem{cor}[theo]{Corollary}
\newtheorem{defn}[theo]{Definition}
\newtheorem{cond}[theo]{Condition}

\DeclareMathOperator{\R}{\mathbb{R}}

\DeclareMathOperator{\N}{\mathbb{N}}

\DeclareMathOperator{\E}{\mathbb{E}}

\DeclareMathOperator{\Prob}{\mathbb{P}}

\DeclareMathOperator{\sN}{\mathcal{N}}

\newcommand{\vertiii}[1]{{\left\vert\kern-0.25ex\left\vert\kern-0.25ex\left\vert #1 
    \right\vert\kern-0.25ex\right\vert\kern-0.25ex\right\vert}}

\renewcommand{\epsilon}{\varepsilon}

\author{Sascha Troscheit\thanks{The author was supported NSERC Grants 2016\,03719, RGPIN-2014-03154 and the Department of Pure Mathematics at the University of Waterloo.}}
\date{
	\emph{\small Department of Pure Mathematics, University of Waterloo, Waterloo, ON, N2L 3G1, Canada}\\[2ex]
	\today}
\title{The quasi-Assouad dimension for stochastically self-similar sets}

\begin{document}

\maketitle

\begin{abstract}
	The class of stochastically self-similar sets contains many famous examples of random sets, \eg Mandelbrot percolation and general fractal percolation. Under the assumption of the uniform open set condition and some mild assumptions on the iterated function systems used, we show that the quasi-Assouad dimension of self-similar random recursive sets is almost surely equal to the almost sure Hausdorff dimension of the set. 
	We further comment on random homogeneous and $V$-variable sets and the removal of overlap conditions.
\end{abstract}

\vspace{1cm}

\section{Introduction}
The Assouad dimension was first introduced by Patrice Assouad in the 1970's to solve embedding problems, see \cite{assouadphd} and \cite{Assouad79}. It gives a highly localised notion of `thickness' of a metric space and its study in the context of fractals has attracted a lot of attention in recent years. We point the reader to  \cite{Angelevska16,Fraser15a,Fraser14a,Garcia17,Garcia16,Kaenmaki16,Luukkainen98,Lu16,Mackay11,Olsen11,Olson16} for many recent advances on the Assouad dimension.

Let $F$ be a totally bounded metric space and let $0<r\leq R\leq\lvert F\rvert$, where $\lvert.\rvert$ denotes the diameter of a set. Let $N_r(X)$ be the minimal number of sets of diameter $r$ necessary to cover $X$. 
We write $N_{r,R}(F)=\max_{x\in F}N_r(B(x,R)\cap F)$ for the minimal number of centred open $r$ balls needed to cover any open ball of $F$ of diameter less than $R$.
Let 
\[
h(\delta,F)=\inf\{\alpha \mid \exists C>0,\;\forall0<r\leq R^{1+\delta}\leq\lvert F\rvert\text{ we have }N_{r,R}(F)\leq C(R/r)^\alpha\}.
\]
The Assouad dimension is given by $\dim_A(F)=h(0,F)$; it is the minimal exponent such that all open balls of $F$ can be covered by a certain number of $r$ balls relative to the size of the ball of $F$.
We note that $\delta=0$ gives no restriction on the ratio $R/r$ other than that it is greater than $1$.
For positive $\delta>0$ this means, however, that there must be a gap between $r$ and $R$ that grows as $R$ decreases. 
We immediately conclude that $h(\delta,F)\leq h(\delta',F)$ for $0\geq\delta'<\delta$. 
It was shown in \cite{Garcia17} that $h$ may not be continuous in $\delta$ at $0$ and the quasi-Assouad dimension is defined by the limit $\dim_{qA}F=\lim_{\delta\to 0}h(\delta,F)$. 
This quantity was first introduced by L\"u and Xi in~\cite{Lu16} and differs from the Assouad dimension by ignoring values of $r$ close to $R$.
It was also shown in \cite{Garcia17} that in some self-similar deterministic settings this gap is sufficient for the quasi-Assouad dimension to coincide with the upper box-counting dimension irrespective of overlaps. 
Recently, Fraser and Han~\cite{Fraser16b,Fraser16a} introduced the notion of the Assouad spectrum, which is defined similarly. Here, the relative size of $r$ and $R$ is fixed by a quantity $\theta\in(0,1)$ and we restrict covers to $r=R^{1/\theta}$.
We define
\[
\dim_A^\theta(F)=\inf\{\alpha \mid \exists C>0,\;\forall0<r^\theta = R\leq\lvert F\rvert\text{ we have }N_{r,R}(F)\leq (R/r)^\alpha\}.
\]

For totally bounded metric spaces $F$ the following inequalities are immediate,
\begin{equation}
\dim_H F\leq \overline\dim_B F \leq \dim_A^\theta F\leq \dim_{qA}F\leq\dim_A F\quad(\text{for all }\theta\in(0,1)).\label{eq:basicIneq}
\end{equation}


Dimension theoretic results usually assume some separation condition and here we will use the uniform open set condition. If one assumes that all iterated function systems are self-similar and they satisfy the uniform open set condition, one can easily determine the almost sure Hausdorff dimension of the associated stochastically self-similar attractors.
Let $\mathbb{L}=\{\mathbb{I}_{\lambda}\}_{\lambda\in\Lambda}$ be a collection of iterated function systems indexed by $\lambda\in\Lambda$.
Let $c_\lambda^i$ be the contraction rate of the map $f_\lambda^i$. Then the almost sure Hausdorff dimension of random recursive sets is unique $s$ satisfying
\[
\E\left( \sum_{j=1}^{\#\mathbb{I}_{\lambda}}(c_{\lambda}^{j})^{s}\right)=1,
\]
with $\lambda$ chosen according to some `nice' probability measure on $\Lambda$, see \eg \cite{Graf87, Falconer86}. 
Further, Troscheit~\cite{Troscheit15} proved that the Hausdorff, box-counting, and packing dimensions all coincide for these sets, while the Assouad dimension is `maximal' in some sense. 
For example, the limit set of Mandelbrot percolation of the $d$-dimensional unit cube for supercritical probabilities has Assouad dimension $d$, conditioned on non-extinction, see also Berlinkov and J\"arvenp\"a\"a~\cite{Berlinkov16}; and Fraser, Miao, and Troscheit~\cite{Fraser14c} for earlier results. 
The notion of Assouad spectrum, $\dim_A^\theta$ for $\theta\in(0,1)$, was applied to Mandelbrot percolation by Fraser and Yu \cite{Fraser16b,Fraser16a} to obtain more information about its scaling. 
Surprisingly, they found that the Assouad spectrum is constant and equal to the box-counting dimension for all $\theta\in(0,1)$ but their result relies on technical and difficult estimates. 

In this paper we show that this unexpected result is -- in fact -- a general feature of all random recursive sets that satisfy the uniform open set condition and that the quasi-Assouad dimension is almost surely equal to the box-counting (and hence Hausdorff) dimension of random recursive attractors. We use fundamental results about Galton--Watson processes to provide a simple proof of this very general result. 
To achieve this we first obtain an approximation that allows us to talk about a single contraction ratio rather than possibly uncountably many.
We then estimate and bound the number of overlapping intervals which allows us to focus solely on an associated Galton--Watson process and we show that there cannot be `too many' descendants. This in turn gives us bounds for the quasi-Assouad dimension sufficient to prove almost sure equality with the box-counting dimension. 
The inequalities (\ref{eq:basicIneq}) then imply the constant spectrum as a corollary.

\section{Random recursive sets and main results}

Let $\Lambda$ be a compact subset of $\R^k$ for some $k$ and let $\mu$ be a compactly supported Borel probability measure supported on $\Lambda$. 
We use $\Lambda$ to index our choices of iterated function systems, which we assume to have cardinality up to some $\sN\in\N$. 
We denote the \emph{iterated function system (IFS)} associated with $\lambda$ by $\mathbb{I}_\lambda=\{f_\lambda^1,f_\lambda^2, \dots , f_\lambda^{\sN(\lambda)}\}$, where $\sN(\lambda)\leq \sN$ and $f_i^j$ are similarities.
If $0\in\Lambda$, then we assume without loss of generality that $0$ is an isolated point (in $\Lambda$) and set $\mathbb{I}_0=\varnothing$. We further assume that $0$ is the only point in $\Lambda$ where the associated IFS is empty. 
Let $\mathbb{L}=\{\mathbb{I}_\lambda\}_{\lambda\in\Lambda}$. The \emph{random iterated function system (RIFS)} is the tuple $(\mathbb{L},\mu)$.

We use the construction of random code-trees, introduced by J\"arvenp\"a\"a et al.~\cite{Jarvenpaa14a,Jarvenpaa16,Jarvenpaa17} and in Section~\ref{sect:examples} we list several examples and implementations of code-trees with the intent to make the abstract concept of random code-trees more accessible.
Consider the rooted $\sN$-ary tree $\mathfrak{T}$. We label each node with a single $\lambda\in\Lambda$, chosen independently, according to probability measure $\mu$. 
We denote the space of all possible labellings of the tree by $\mathcal{T}$ and refer to individual realisations by $\tau\in\mathcal{T}$.\label{defn:randomRecRealisation}
In this full tree we address vertices by which branch was taken; if $v$ is a node at level $k$ we write $v=(v_{1},v_{2},\dots,v_{k})$, with $v_{i}\in\{1,\dots,{\sN}\}$ and root node $v=(.)$.
The set of all tree levels $\mathfrak{T}$ is then:
\[
\mathfrak{T}=\{\{(.)\}, \{(1),(2),\dots,(\sN)\}, \{ (1,1),(1,2),\dots,(1,\sN),(2,1),\dots,(\sN,\sN)\},\dots\;\}.
\]
Alternatively, we can consider $\tau:\mathfrak{T}\to\Lambda$ as a function assigning labels to nodes, \ie $\tau(v)\in\Lambda$ is the label of the node $v$.
Given a node $v$ we define $\sigma^{v}\tau$ to be the full subtree starting at vertex $v$, with $\sigma^{(.)}\tau=\tau$. 
There exists a natural measure $\Prob$ on the collection of trees induced by $\mu$. Informally, this measure is obtained by assigning a label $\lambda\in\Lambda$ to each node of $\mathfrak{T}$ independently according to $\mu$. 
Formally, let $V\subset\mathfrak{T}$ be a finite collection of vertices. To each $v\in V$ we associate an open set $O(v)\subseteq\Lambda$ and define $\widetilde{V}=\{\tau\in\mathcal{T} \mid \tau(v)\in O(v)\text{ for all }v\in V\}$.
We define $\Prob$ for each of these possible pairings by
\[
\Prob(\widetilde{V})=\prod_{v\in V} \mu (O(v)).
\]
We note that the collection of all $\widetilde{{V}}$ is a basis for the topology of $\mathcal{T}$ and, applying Caratheodory's extension principle, $\Prob$ extends to a unique Borel measure on $\mathcal{T}$.

We write $e_{\lambda}^{j}$ for the letter representing the map $f_{\lambda}^{j}\in\mathbb{I}_{\lambda}$. By assumption we take $\mathbb{I}_0=\varnothing$ and use the letter $\emptyset$ to represent the \emph{empty map}. For each full tree $\tau$ that is labelled by entries in $\Lambda$, we associate another rooted labelled $N$-ary tree $\mathbf{T}_{\tau}$ to the realisation $\tau$, where each node is labelled by a `coding' describing a composition of maps. Given two codings $e_1$ and $e_2$, we write $e_1 e_2 = e_1 \odot e_2$ for concatenation. We let $\epsilon_0$ be the \emph{empty word} and use the letter $\emptyset$ as a multiplicative zero, \ie $\emptyset\odot e = e\odot\emptyset=\emptyset$, to represent the extinction upon using the empty IFS. Similarly, if $\{e_i\}$ is a collection of codings, then $\{e_i\}\cup\emptyset=\{e_i\}$.

\begin{defn}\label{infinityDef1}
	Let $\mathbf{T}_{\tau}$ be a labelled tree.
	We write $\mathbf{T}_{\tau}(v)$ for the label of node $v$ of the tree $\mathbf{T}_{\tau}$.
	The \emph{coding tree} $\mathbf{T}_{\tau}$ is then defined inductively:
	\[
	\mathbf{T}_{\tau}((.))=\epsilon_{0}\text{ and }
	\mathbf{T}_{\tau}(v)=\mathbf{T}_{\tau}((v_{1},\dots,v_{k}))=\mathbf{T}_{\tau}((v_{1},\dots,v_{k-1}))\odot e_{\tau(v_{k-1})}^{v_{k}}
	\]
	for $1\leq v_{k}\leq \mathscr{N}_{\tau(v_{k-1})}$ and $e_{\tau(v_{k-1})}^{v_{k}}=\emptyset$ otherwise. This `deletes' this subbranch as $\emptyset$ annihilates under multiplication.
		We refer to the the set of all codings at the $k$-th level by 
	\[
	\mathbf{T}_{\tau}^{k}=\bigcup_{1\leq v_{1},\dots,v_{k}\leq \sN}\mathbf{T}_{\tau}((v_{1},\dots,v_{k})).
	\]
\end{defn}

To transition from these coding spaces to sets we stipulate some conditions on the maps $f_\lambda^j$.
\begin{cond}\label{cond:mappingconditions}
	Let $(\mathbb{L},\mu)$ be a RIFS. We assume that all maps $f_\lambda^i:\R^d\to\R^d$ are similarities, \ie 
	$\lVert f_\lambda^i(x)-f_\lambda^i(y)\rVert=c_\lambda^i \lVert x-y\rVert$, where $\lVert.\rVert$ denotes the Euclidean metric. Further, they are strict contractions with contraction rates uniformly bounded away from $0$ and $1$, \ie there exist $0< c_{\min} \leq c_{\max}<1$ such that $c_{\min} \leq c_\lambda^i \leq c_{\max}$ for all $\lambda\in \Lambda\setminus\{0\}$ and $1\leq i \leq \sN(\lambda)$.
\end{cond}

We can now define the random recursive set.
\begin{defn}\label{infinityDef2}
	Let $({\mathbb{L}},\mu)$ be a RIFS and $\tau\in\mathcal{T}$. The \emph{random recursive set} $F_{\tau}$ is the compact set satisfying
	\[
	F_{\tau}=\bigcap_{k=1}^{\infty}\bigcup_{e\in\mathbf{T}_{\tau}^{k}} f_{e_1}\circ f_{e_2} \circ \dots \circ f_{e_k}(\Delta),
	\]
	where $\Delta$ is a sufficiently large compact set, satisfying $f_{e_\lambda^i}(\Delta)\subseteq\Delta$ for all $\lambda\in\Lambda$ and $1\leq i\leq \sN(\lambda)$.
\end{defn}

We end by noting that from a dimension theoretical perspective, the subcritical and critical cases are not interesting. In those cases, almost surely, only a finite number of branches will survive. Since the associated maps are strict contractions, the resulting random set will be discrete and all dimensions coincide trivially. 
The supercritical case ensures that there exists positive probability that there is an exponentially increasing number of surviving branches; we will focus on this case and stipulate the following condition.
\begin{cond}
	Let $(\mathbb{L},\mu)$ be a RIFS. We say that $(\mathbb{L},\mu)$ is \emph{non-extinguishing} if 
	\[
	\E(\sN(\lambda))=\int_\Lambda \sN(\lambda)\;d\mu>1.
	\]
	In particular, this means there exists $\epsilon>0$ such that $\mu(\{\lambda\in\Lambda \mid \sN(\Lambda)>1\})\geq\epsilon$ and that $\mu(\{\lambda\in\Lambda\mid \sN(\lambda)\geq 1\})=1-\mu(\{0\})$.
\end{cond}
\subsection{Main Result}

To give meaningful dimension results some assumptions have to be made on the amount of overlap of the images.
\begin{defn}
	Let $(\mathbb{L},\mu)$ be a RIFS. If there exists an open set $\mathcal{O}\subset\R^d$ satisfying
	\begin{enumerate}
		\item for all $\lambda\in\Lambda\setminus\{0\}$ and $1\leq i\leq \sN(\lambda)$ we have $f_\lambda^i (\mathcal{O})\subseteq\mathcal{O}$, and
		\item for all $\lambda\in\Lambda\setminus\{0\}$ and $1\leq i\leq \sN(\lambda)$, if $f_\lambda^i (\mathcal{O})\cap f_\lambda^j(\mathcal{O})\neq\varnothing$ then $i = j$,
	\end{enumerate}
	we say that the RIFS satisfies the \emph{uniform open set condition (UOSC)}.
\end{defn}

\begin{theo}[Main Theorem]\label{thm:mainTheorem}
	Let $(\mathbb{L},\mu)$ be a non-extinguishing RIFS that satisfies the UOSC and Condition~\ref{cond:mappingconditions}.
	Then,  for $\Prob$-almost every $\tau\in\mathcal{T}$,
	\[
	\dim_{qA}F_\tau=\dim_H F_\tau.
	\]
\end{theo}
We immediately obtain the following corollary.
\begin{cor}
	Let $F_\tau$ be as above. The dimension spectrum is constant and coincides with the box-counting dimension, \ie for all $\theta\in (0,1)$ and almost every $\tau\in\mathcal{T}$,
	\[
	\dim_B F_\tau = \dim_A^\theta F_\tau.
	\]
\end{cor}
\begin{proof}
	The Assouad spectrum is bounded below by the Hausdorff dimension and above by the quasi-Assouad dimension, see (\ref{eq:basicIneq}). Since $\dim_{qA}F_\tau=\dim_H F_\tau$ with full probability, the Assouad spectrum must also coincide with these values.
\end{proof}

\section{Proof of Main Theorem}
Before we prove the main theorem we prove some constructive lemmas and define a stopping set where the contraction rates become roughly comparable. Recall that $\mathcal{O}$ is the open set guaranteed by the uniform open set condition, with topological closure $\overline{\mathcal{O}}$, and that $\lvert.\rvert$ is the diameter of a set.
\begin{defn}
	Let $\epsilon>0$. The set of all codings that have associated contraction of rate comparable to $\epsilon$ and do not go extinct are denoted by
	\[
	\Xi_{\epsilon}(\tau)=\{e\in\mathbf{T}_\tau^k \mid k\in\N \text{ and } \lvert f_{e_1}\circ\dots\circ f_{e_k}(\overline{\mathcal{O}})\rvert<\epsilon\leq\lvert f_{e_1}\circ\dots\circ f_{e_{k-1}}(\overline{\mathcal{O}})\rvert \}\cap \Sigma,
	\]
	where
	\[
	\Sigma=\left\{e\in\bigcup_{k\in\N}\mathbf{T}_\tau^k \;\Big|\;  \text{for all } m>k \text{ there exists } \hat{e}\in \mathbf{T}_\tau^m \text{ such that }e_i=\hat e_i \text{ for }1\leq i \leq k \text{ and }\hat e_i\neq\emptyset\right\}.
	\]
	We refer to $\Xi_\epsilon (\tau)$ as the \emph{$\epsilon$-codings} and $\Sigma$ as the \emph{non-extinguishing codings}.
	
\end{defn}

We can limit the number of overlapping cylinders using the uniform open set condition.
\begin{lma} \label{upperasslem1}
	Assume that $(\mathbb{L},\mu)$ satisfies the UOSC. 
	Then
	\[
	\#\{e\in\Xi_{\epsilon}(\tau)\mid \overline{f_e(\mathcal{O})}\cap B(z,r)\neq\varnothing\} \, \leq \,  (4/c_{\min})^{d}
	\]
	for all $z\in F_{\tau}$ and  $\epsilon \in (0,1]$, where $\mathcal{O}$ is the open set guaranteed by the UOSC.
\end{lma}

We note that this proof is based heavily on \cite{Hutchinson81} and \cite{Olsen11}, and first appeared in a simplified version in~\cite{Fraser14c}. This version, allowing for RIFS with infinitely many IFSs, was first given in the author's PhD thesis~\cite[Lemma 5.1.5]{TroscheitPhDThesis} and we reproduce it here.

\begin{proof}
	Fix $z\in F_{\tau}$ and $\epsilon>0$.  Let ${\Xi}=\{e\in\Xi_{\epsilon}(\tau)\mid \overline{f_e(\mathcal{O})}\cap B(z,\epsilon)\neq\varnothing\}$ and suppose the ambient space is $\R^{d}$.\index{stopping} We have
	\[
	\#\Xi(\epsilon c_{\min})^{d} \ = \ \sum_{e\in\Xi}(\epsilon c_{\min})^{d} \ \leq \ \sum_{e\in\Xi}\lvert\overline{f_e(\mathcal{O})}\rvert^{d}.
	\]
	But since $\overline{f_e(\mathcal{O})}\cap B(z,\epsilon)\neq\varnothing$ and $|\overline{f_e(\mathcal{O})}|<r$ we find $f_e(\mathcal{O})\subseteq B(z,2\epsilon)$ for all $e \in \Xi$ and since the sets $f_e(\mathcal{O})$ are pairwise disjoint by the UOSC we have
	\[
	\#\Xi(\epsilon c_{\min})^{d} \ \leq \ \sum_{e\in\Xi}|f_e(\mathcal{O})|^{d} \ \leq \  \mathcal{L}^{d}(B(z,2\epsilon)) \ \leq \ (4\epsilon)^{d},
	\]
	where $\mathcal{L}^{d}$ is the $d$-dimensional Lebesgue measure.  It follows that $\#\Xi\leq(4/c_{\min})^{d}$.
\end{proof}

Clearly the projections of these $\epsilon$-codings are a cover of $F_\tau$, \[\bigcup_{e\in\Xi_{\epsilon}(\tau)}f_e(\overline{\mathcal{O}})\supseteq F_\tau,\]and so $N_\epsilon(F_\tau) \leq \#\Xi_{\epsilon}(\tau)$. Now every $\epsilon$-coding has at least one descendent and by Lemma~\ref{upperasslem1} there cannot be more than $(4/c_{\min})^d$ many $\epsilon$-codings such that their images intersect any ball of radius $\epsilon$. Therefore there exists $C>0$ independent of $\tau$ and $\epsilon$ such that 
\begin{equation}\label{eq:comparable}
C\#\Xi_{\epsilon}(\tau)\leq N_\epsilon(F_\tau) \leq \#\Xi_{\epsilon}(\tau).
\end{equation}

Further, 
\begin{equation}\label{eq:approximationBounding}
\sN^{-1}\sum_{e\in \Xi_{\epsilon}(\tau)}\#\Xi_{\delta}(\sigma^{e}\tau)\leq \#\Xi_{\epsilon\delta}(\tau)\leq \sN\sum_{e\in \Xi_{\epsilon}(\tau)}\#\Xi_{\delta}(\sigma^{e}\tau).
\end{equation}
Briefly, this is because joining appropriate $\delta$-codings to $\epsilon$-codings might not be $\epsilon\delta$-codings but their numbers may at most differ by a multiple of the maximal number of descendants, $\sN$.
Inductively, we obtain
\begin{equation}\label{eq:approximationBoundingMod}
\#\Xi_{\epsilon^k}(\tau)\leq \sN^k\sum_{e_1\in \Xi_{\epsilon}(\tau)}\;\;\sum_{e_2\in \Xi_{\epsilon}(\sigma^{e_1}\tau)} \dots \sum_{e_k\in \Xi_{\epsilon}(\sigma^{e_{k-1}}\tau)}\#\Xi_{\delta}(\sigma^{e_k}\tau)
\end{equation}
from (\ref{eq:approximationBounding}) with an analogous lower bound.

By definition the summands in (\ref{eq:approximationBounding}) and (\ref{eq:approximationBoundingMod}) are i.i.d.\ random variables and, fixing $\epsilon>0$, we write $X^\epsilon= X^\epsilon(\tau) = \#\Xi_{\epsilon}(\tau)$ for a generic copy of this random variable. We observe that $X^\epsilon$ is a random variable taking values from $0$ to some $M\in \N$ with fixed probability and we consider the Galton--Watson process with that offspring distribution.
It is a basic result (see, \eg~\cite{AthreyaBook}) that the Galton--Watson process
\[
X^\epsilon_k = \sum_{j=1}^{X^\epsilon_{k-1}} X^\epsilon
\]
satisfies $\log X_k^\epsilon / k \to \E(X^\epsilon)$ as $k\to\infty$ almost surely.
In particular there exists a unique $s_\epsilon$ such that $\E(X)=\E(\#\Xi_{\epsilon})=\epsilon^{-s_\epsilon}$. We note that
\begin{align*}
s_B&=\lim_{\epsilon\to0}  \frac{\log N_{\epsilon}F(\tau)}{\log(1/\epsilon)}=
\lim_{\epsilon\to0} \lim_{k\to\infty} \frac{\log N_{\epsilon^k}(F_\tau)}{\log(1/\epsilon^k)}=
\lim_{\epsilon\to0} \lim_{k\to\infty} \frac{\log \#\Xi_{\epsilon^k}(\tau)}{\log(1/\epsilon^k)}\\
&\leq\lim_{\epsilon\to 0}\lim_{k\to\infty} \frac{\log(\sN^k X_k)}{k\log(1/\epsilon)}
=\lim_{\epsilon\to 0}\left[ \frac{\log \sN}{\log(1/\epsilon)}+\frac{\E X}{\log(1/\epsilon)}\right]=\lim_{\epsilon\to0}s_\epsilon.
\end{align*}
The opposite inequality can be achieved by using $\sN^{-1}$ in the lower bound and we conclude $s_B=\lim_{\epsilon\to0}s_\epsilon$.
This means that these approximations give us a behaviour as close to the limit behaviour as one desires and we will now analyse the associated Galton--Watson process. 

\begin{lma}[Athreya~{\cite[Theorem 4]{Athreya94}}]\label{lma:athreyabound}
	Let $X_k^\epsilon$ be a Galton--Watson process with mean $m=\E(X^\epsilon)<\infty$.
	Suppose\footnote{We note that in \cite{Athreya94} the expectation is conditioned on the first generation being $1$. In our case we always assume a rooted tree making this conditioning superfluous. Further, their paper mostly considers the case where the probability of having no descendent is zero. However, this condition is not used in this particular statement.}
	that $\E(\exp(\theta_0 X^\epsilon))<\infty$ for some $\theta_0>0$. Then there exists $\theta_1>0$ such that
	\[
	\sup_k \E(\exp(\theta_1 W_k^\epsilon))<\infty,
	\]
	where $W_k^\epsilon=X_k^\epsilon/m^k$ is the associated normalised Galton--Watson process.
\end{lma}
We can use this important result to prove this immediate lemma.
\begin{lma}\label{lma:superexponentialdecay}
	Let $X_k^\epsilon$ be a Galton--Watson process with mean $m=\E (X^\epsilon)<\infty$ such that $C_1=\sup_{\tau}X^\epsilon(\tau)<\infty$. Let $C>0$ be some constant. There exist $t>0$ and $D>0$ such that
	\[
	\Prob\left\{X_k^\epsilon\geq C m^{(1+\epsilon)k}\right\}\leq D e^{- t m^{\epsilon k}},
	\]
	\ie the probability that $X_k^\epsilon$ exceeds $C m^{(1+\epsilon)k}$ decreases superexponentially in $k$.
\end{lma}
\begin{proof}
	
	Let $W_k^\epsilon=X_k / m^k$. Then $\E(W_k)=1$ for all $k$. 
	Since $C_1<\infty$, we conclude that $\E(\exp(X^\epsilon))\leq \exp{C_1}<\infty$ and the conditions of Lemma~\ref{lma:athreyabound} are satisfied for $\theta_0=1$. Therefore exists $\theta_1>0$ and $D>0$ such that $\E(\exp(\theta_1 W_k))\leq D$ for all $k$.
	We use the Chernoff bound to obtain
	\begin{align*}
	\Prob\left\{X_k\geq C m^{(1+\epsilon)k}\right\}&=\Prob\left\{ W_k \geq C m^{\epsilon k}\right\}\\
	&= \Prob\left\{ \exp(\theta_1 W_k) \geq \exp(C\theta_1 m^{\epsilon k})\right\}\\
	&\leq \frac{\E\exp(\theta_1 W_k)}{\exp(C\theta_1 m^{\epsilon k})}.
	\end{align*}
	 So, using Lemma~\ref{lma:athreyabound},
	\[
	\Prob\left\{X_k\geq C m^{(1+\epsilon)k}\right\}\leq D e^{- t m^{\epsilon k}},
	\]
	for $t=C\theta_1>0$ as required.
\end{proof}

Given a starting vertex, we can thus bound the probability that the process will eventually exceed $m^{(1+\epsilon)k}$.

\begin{cor}\label{cor:superexponentialsum}
	Let $X_k^\epsilon$ be a Galton--Watson process with mean $m=X^\epsilon<\infty$ such that $\sup_{\tau}X^\epsilon(\tau)<\infty$. Assume $C>0$ is some constant, then 
	\[
	\Prob\{X_k^\epsilon>C m^{(1+\epsilon)k} \text{ for some }k\geq l\}\leq \sum_{k=l}^\infty D e^{- t m^{\epsilon k}}
	\leq D' e^{-t m^{\epsilon l}}
	\]
	for some $D'>0$ independent of $l$.
\end{cor}

We are now ready to prove Theorem~\ref{thm:mainTheorem}.
\begin{proof}[Proof of Main Theorem]
	Fix $\epsilon>0$ and let $\eta>0$ be small enough such that $-\log \sN / \log \eta<\epsilon / 4$. Consider the probability that the Assouad dimension exceeds the almost sure box-counting dimension $s=\dim_B F_\tau$.
	\begin{align}
	\Prob\left\{\dim_{qA}F_\tau \geq s+\epsilon \right\} &\leq 
	\Prob\Big\{\forall\delta>0,\;\exists (x_i,r_i,R_i)_{i\in\N}\in(F_\tau\times\R^+\times\R^+)^{\N}\text{ such that }r_i\leq R_i^{1+\delta},\nonumber\\
	&\hspace{1.0cm}R_i\to 0\text{ as }i\to\infty\text{ and }N_{r_i}(B(x_i,R_i)\cap F_\tau)\geq \left(\frac{R_i}{r_i}\right)^{s+\epsilon/2}\Big\}\label{eq:probabilityToBound}.
	\end{align}
	By Lemma~\ref{upperasslem1} the number of words comparable to $R_i$ that have non-trivial intersection with $B(x_i,R_i)\cap F_\tau$ is bounded above by $(4/c_{\min})^d$ and so we can get a new upper bound to (\ref{eq:probabilityToBound}).
	\begin{align}
	&\leq \Prob\Big\{ \forall \delta>0, \; \exists (r_i,k_i)_i\in(\R^+\times\N)^{\N}, \; \exists v\in\mathbf{T}^{k_i}(\tau)\text{ such that }r_i\leq \lvert f_v (F_{\sigma^v \tau})\rvert^{1+\delta},\;k_{i+1}>k_i\nonumber\\
	&\hspace{6cm}\text{ and }N_{r_i}(f_v (F_{\sigma^v \tau}))\geq (c_{\min}/4)^d \left(\frac{\lvert f_v (F_{\sigma^v \tau})\rvert}{r_i}\right)^{s+\epsilon/2}\Big\}\nonumber
	\end{align}
	Now $N_{r_i}(f_v (F_{\sigma^v \tau}))\leq \sN\cdot N_{r_i/d_v}(F_{\sigma^v \tau})\leq \sN \cdot \#\Xi_{r_i/d_v}(\sigma^v \tau)$ by (\ref{eq:comparable}), where $d_v=\lvert f_v (F_{\sigma^v \tau})\rvert$. 
	Expressing this in terms of the Galton--Watson process for approximation $\eta$, we obtain 
	\[N_{r_i}(f_v (F_{\sigma^v \tau}))\leq \sN^{k_v+1} X^\eta_{k_v}(\sigma^v\tau),\] where ${k_v}$ is such that $\eta^{{k_v}}\leq r_i/d_v<\eta^{k_v-1}$. Therefore
	\begin{align}
	\Prob\left\{\dim_{qA}F_\tau \geq s+\epsilon \right\} &\leq 
	\Prob\Big\{ \forall \delta>0 \text{ there exists }(l_i,k_i)_{i\in\N}\in(\N\times\N)^{\N},\;\exists v\in\mathbf{T}^{l_i}(\tau)\nonumber\\
	&\hspace{0.5cm}\text{ such that }\eta^{k_i}\leq d_v^\delta \text{ and }X^\eta_{k_i}(\sigma^v\tau)\geq \frac{c_{\min}^d}{4^d \sN^{k_i+1}}\left(\eta^{-k_i}\right)^{s+\epsilon/2}\Big\}\label{eq:borel-cantellipre}.
	\end{align}
	Now, fix $\delta>0$. We now estimate $\Prob\{ X^\eta_{k}(\tau)\geq C{\sN^{-k}}\eta^{-k(s+\epsilon/2)} \}$, where $C>0$ is a uniform constant. First, by the choice of $\eta$ we have $C{\sN^{-k}}\eta^{-k(s+\epsilon/2)}\leq C \eta^{-k(s+\epsilon/4)}$ and so 
	\[
	\Prob\{ X^\eta_{k}(\tau)\geq C{\sN^{-k}}\eta^{-k(s+\epsilon/2)}\}\leq \Prob\{ X^\eta_{k}(\tau)\geq C\eta^{-k(s+\epsilon/4)}\}.
	\]
	Now consider $W_k^\eta=X_k^\eta/m_\eta^k$ for $m_\eta=\E(X^\eta)$. Clearly $m_\eta\leq \eta^{-s}=m$, and so 
	$\E(X^\eta_{k})\leq \eta^{-sk}$. We can now apply Lemma~\ref{lma:superexponentialdecay} and Corollary~\ref{cor:superexponentialsum} to obtain
	\[
	\Prob\{ X^\eta_{k}(\tau)\geq C{\sN^{-k}}\eta^{-k(s+\epsilon/2)}\}\leq D e^{-tm^{\epsilon k/4}}
	\]
	and
	\[
	\Prob\{ X^\eta_{k}(\tau)\geq C{\sN^{-k}}\eta^{-k(s+\epsilon/2)}\text{ for some }k\geq l\}\leq D' e^{-tm^{\epsilon l/4}}.
	\]
	Now consider again the event in equation (\ref{eq:borel-cantellipre}). First note that $d_v\leq c_{\min}^{l_i}$ and thus the event with $\eta^{k_i}\leq d_v^\delta$ replaced by $\eta^{k_i}\leq c_{\min}^{\delta l_i}$ has greater probability. The last condition is equivalent to $k_i\geq \delta l_i \log c_{\min}/\log\eta=\delta' l_i$, for some $\delta'>0$ depending on $\eta$ and $\delta$.
	So,
	\begin{align}
	\Prob\left\{\dim_{qA}F_\tau \geq s+\epsilon \right\} &\leq\sup_{\delta>0}\Prob\{\text{ there exist infinitely many }l\in\N,\;\exists v\in\mathbf{T}^{l}(\tau),\nonumber\\
	&\hspace{3.5cm}\exists k\geq \delta' l\text{ such that }
	 X^\eta_{k}(\tau)\geq C{\sN^{-k}}\eta^{-k(s+\epsilon/2)}\}.\label{eq:lasboundingThing}
	\end{align}
	But for fixed $\delta$ and $\l$,
	\begin{equation}
	\Prob\{\;\exists v\in\mathbf{T}^{l}(\tau),\;\exists k\geq \delta' l\text{ such that }
	X^\eta_{k}(\tau)\geq C{\sN^{-k}}\eta^{-k(s+\epsilon/2)}\}
	\leq\sN^l D' \exp\left(-tm^{\epsilon \delta'l/4}\right).\nonumber
	\end{equation}
	But since
	\begin{multline}
	\sum_{l\in\N}\Prob\{\;\exists v\in\mathbf{T}^{l}(\tau),\;\exists k\geq \delta' l\text{ such that }
	X^\eta_{k}(\tau)\geq C{\sN^{-k}}\eta^{-k(s+\epsilon/2)}\}\\
	\leq\sum_{l=1}^\infty \sN^l D' \exp\left(-tm^{\epsilon \delta'l/4}\right)
	=\sum_{l=1}^\infty D' \exp\left(l\log\sN-tm^{\epsilon \delta'l/4}\right)<\infty, \nonumber
	\end{multline}
	we conclude that the right hand side of (\ref{eq:lasboundingThing}) is zero for every $\delta>0$ by the Borel--Cantelli Lemma. Hence $\Prob\left\{\dim_{qA}F_\tau \geq s+\epsilon \right\}=0$ and so, by arbitrariness of $\epsilon$, the claim is proven.
\end{proof}

\section{Other random models and overlaps}
\subsection{The random homogeneous and $V$-variable case}
Random recursive sets are not the only natural way of defining random sets with varying iterated function systems.
Another important model is the random homogeneous model, also referred to as the $1$-variable model. This is in reference to the more general $V$-variable model which we will also discuss in this section.

The random homogeneous model can --informally-- be described as applying the same IFS at every stage of the construction. In keeping with our flexible notation, the random iterated function system $(\mathbb{L},\mu)$ has an associated random homogeneous model that is also defined as the projection of randomly chosen $\tau\in\mathcal{T}$ albeit with a different measure $\Prob_1$. This measure is supported on the subset $\mathcal{T}_1=\{\tau\in\mathcal{T} \mid \tau(v)=\tau(w) \text{ whenever } d(v)=d(w)\}$, where $d(v)$ is the tree depth at which $v$ occurs.
Let $L\subset\N$ be a finite set and associate an open set $O(l)\subseteq\Lambda$ to each of these integers. Write $1_l=(1,\dots,1)$ for the node consisting of $l$ many $1$s. The measure $\Prob_1$ is defined on all $\widetilde{L}=\{\tau\in\mathcal{T}_1 \mid \tau(1_l)\in O(l)\text{ for all }l\in L\}$ by
\[
\Prob_1(\widetilde{L})=\prod_{l\in L} \mu(O(l)).
\]
We refer the reader to \cite{Hambly92} and \cite{Troscheit15} for more information on these sets.

The class of $V$-variable attractors were first introduced by Bransley et al.~in \cite{Barnsley05,Barnsley08,Barnsley12} with the aim to model more complicated natural processes. It is characterised by allowing up to $V$ different behaviours at every level of the construction and they can be similarly defined with the notation of code-trees. First we define the subset of $\mathcal{T}$ of interest.
Let 
\[
\mathcal{T}_V=\Big\{\tau\in\mathcal{T} \;\Big| \;\;\sup_l \{\sigma^v \tau \mid v\in\mathfrak{T}\text{ and }d(v)=l\}\leq V\;\Big\},
\]
be the subset of $\mathcal{T}$ such that at every tree depth there are at most $V$ different subtrees. The measure $\Prob_V$ is defined analogously to above, and we end their definition by noting that random homogeneous sets are indeed $V$-variable sets for $V=1$.

In light of our random recursive results one would hope that $V$-variable sets also have the coinciding Hausdorff and quasi-Assouad dimension. This indeed holds, following a strategy close to the one employed in our main theorem. The crucial difference is estimating the number of maximal descendants possible. While the associated process is no longer a Galton--Watson process, it can be modelled by multiplication of positive \emph{i.i.d}.\ variables.
A standard Chernoff bound can be established and the same conclusion achieved. We refer the reader to \cite{Troscheit15,TroscheitPhDThesis} for an explanation of how $V$-variable sets can be modelled by multiplication of \emph{i.i.d}.\ variables and briefly prove the random homogeneous case for illustrative purposes.

\begin{theo}
	Let $(\mathbb{L},\mu)$ be a non-extinguishing RIFS that satisfies the UOSC and Condition~\ref{cond:mappingconditions}. Additionally, assume that $\mu(\{0\})=0$. Then, for $\Prob_1$-almost every $\tau\in\mathcal{T}$,
	\[
	\dim_{qA}F_\tau=\dim_H F_\tau.
	\]
\end{theo}
\begin{proof}
	Similarly to the random recursive case one can show that there exists $C>0$, such that
	\[
	C \sN^{-k} \prod_{i=1}^k \#\Xi_{\epsilon}(\tau_i)\leq N_{\epsilon}(F_\tau)\leq \sN^{k} \prod_{i=1}^k \#\Xi_{\epsilon}(\tau_i),
	\]
	where the $\tau_i$ are chosen independently. Thus we want to find a bound on the probability that this product exceeds its geometric average. Write $Y^\epsilon=\#\Xi_{\epsilon}$ for the generic random variable and $Y_k^\epsilon=\prod_{i=1}^k Y^\epsilon$. Then $\log Y_k^\epsilon\simeq k\log m_\epsilon$ for some $m_\epsilon$ and we find
	\[
	\Prob_1\{\log Y_k \geq (1+\delta)k\log m_\epsilon\}\leq D \exp(-t k)
	\]
	for some $D,t>0$ depending only on $\delta$. The remaining argument now uses the homogeneity to conclude that the probability that one cylinder at level $k$ exceeds the average is equal to the probability that all cylinders exceed the average using the homogeneity. A standard Borel-Cantelli argument then allows the conclusion as in Theorem \ref{thm:mainTheorem}.
\end{proof}

\subsection{Overlaps}

We conjecture that it would be possible to remove the UOSC condition entirely and we will briefly outline why we feel this should be the case.
First we alter our definition of $\epsilon$-codings to take into account the overlaps.
\begin{defn}
	Let $\epsilon>0$. The set of all codings that have associated contraction of rate comparable to $\epsilon$ and do not go extinct are denoted by
	\[
	\Xi_{\epsilon}(\tau)=\{e\in\mathbf{T}_\tau^k \mid k\in\N \text{ and } \lvert f_{e_1}\circ\dots\circ f_{e_k}(\overline{\mathcal{O}})\rvert<\epsilon\leq\lvert f_{e_1}\circ\dots\circ f_{e_{k-1}}(\overline{\mathcal{O}})\rvert \}\cap \Sigma,
	\]
	where
	\[
	\Sigma=\left\{e\in\bigcup_{k\in\N}\mathbf{T}_\tau^k \;\Big|\;  \text{for all } m>k \text{ there exists } \hat{e}\in \mathbf{T}_\tau^m \text{ such that }e_i=\hat e_i \text{ for }1\leq i \leq k\right\}.
	\]
	We refer to $\Xi_\epsilon (\tau)$ as the \emph{$\epsilon$-codings} and $\Sigma$ as the \emph{non-extinguishing codings}.
	We write $S_\epsilon(\tau)$ for the set of all subsets of $\Xi_{\epsilon}(\tau)$, \ie $S_\epsilon(\tau)=\mathcal{P}(\Xi_{\epsilon}(\tau))$ and write
	\[
	S'_\epsilon(\tau)=\{S\in S_\epsilon(\tau) \mid f_v(\overline{\mathcal{O}})\cap f_w(\overline{\mathcal{O}})=\varnothing\text{ for all distinct }v,w\in S\}
	\]
	for those collection of words such that their associated images under $f$ are disjoint.
	Finally we write $\Xi'_{\epsilon}(\tau)\in S'_\epsilon(\tau)$ for the element with maximal cardinality, choosing arbitrarily if it is not unique.
\end{defn}

This altered definition no longer requires Lemma~\ref{upperasslem1} to prove that $N_\epsilon(F_\tau)$ and $\#\Xi'_{\epsilon}(\tau)$ are comparable but allows for a simple geometric argument, see~\cite[Proof of Theorem 3.2.17]{TroscheitPhDThesis}. However, only the lower bound to equations (\ref{eq:approximationBounding}) and (\ref{eq:approximationBoundingMod}) will hold.
This no longer suffices to prove the main result, however, the only difficulty arises in path that were extinct in the $\epsilon$-approximation, but became active later. Our hope was to modify the Galton--Watson process slightly by uniformly including a finite $m$ extra (and normally redundant) paths. That is, we define $X^\eta(\tau)=\#\Xi'_{\eta}(\tau)+m$. For small enough $\eta$, the normalised growth of $X_{k}^\eta$ is less than $s+\epsilon/4$ and we still obtain the same bound from Lemma~\ref{lma:superexponentialdecay} and Corollary~\ref{cor:superexponentialsum}.
It is our hope that these additional redundant paths take up the r\^ole of the codings that `revive' at some point in the process and this would prove the main theorem without our overlap constraints.

\section{Appendix: Examples of random recursive sets}\label{sect:examples}
We now give several examples to illustrate the flexibility of the construction and the tree-codings being used. The first example is a random version of the middle-third Cantor set, where one randomly discards the middle third interval. The second is a Cantor set where three continuous parameters are chosen randomly at every stage. We end by showing that Mandelbrot percolation is a random recursive set for appropriately chosen maps and relate this to general fractal percolation.
\subsubsection*{Example 1}
In its simplest form $\Lambda$ is discrete and $\mu$ is the finite weighted sum of Dirac measures. Let $\Lambda=\{1,2\}$ and $\mu=p\delta_1+(1-p)\delta_2$ for $0<p<1$, where $\delta_x$ is the Dirac measure with unit mass on $x\in\R^k$. Set $\mathbb{I}_1=\{x/2, \; x/2+1/2\}$ and $\mathbb{I}_2=\{x/3,\;x/3+2/3\}$. The first IFS gives rise to the unit interval, whereas the second `generates' the Cantor middle-third set. Since both IFSs have two maps, we set $\sN=2$ and consider the full binary tree.
For any node $v$ in the binary tree, the value $\tau(v)$ is $1$ or $2$ with probability $p$ and $1-p$, respectively, independent of any other node. The resulting set is a subset of the unit line, where at each stage of the construction we either
\begin{itemize}
	\item divide the remaining line segments into halves with probability $p$, or
	\item divide the remaining line segments into thirds and discard the middle interval with probability $1-p$.
\end{itemize} 

We can easily find the almost sure Hausdorff and box packing dimension of the attractor $F_\tau$. It is given by the unique $s$ such that $p(2/2^s)+(1-p)(2/3^s)=1$. It is elementary to show that $\log2/\log3<s<1$ for $0<p<1$. However, its Assouad dimension is almost surely the maximum of the two individual attractors, \ie the Assouad dimension is $1$ almost surely. Our main result show then that the quasi-Assouad dimension and Assouad spectrum equal to the Hausdorff dimension $s$.
\subsubsection*{Example 2}
The construction is flexible enough to allow for more complicated constructions. 
Let $\Lambda=[1/4,1/3]\times[1/4,1/3]\times[1/2,1]\subset\R^3$ and $\mu$ be the normalised Lebesgue measure on $\Lambda$. For $\lambda=(a,b,c)\in\Lambda$ we set $\mathbb{I}_\lambda=\{a x, \; bx+c\}$. This is an infinite family of IFSs that each generate a Cantor set. The translations were chosen such that the UOSC holds with the open unit interval as the open set.
The almost sure Hausdorff dimension is given by the unique $s$ satisfying 
\[1=\E(a^s+b^s)=\frac{8\cdot  3^{-s} - 3 \cdot 2^{1 - 2 s}}{1 + s},\]
and we compute $s\approx 0.56187\dots$.
The Assouad dimension does not fall under the scope of Fraser et al.~\cite{Fraser14c} and Troscheit~\cite{Troscheit15}, however their methods can easily be adapted to show that the Assouad dimension is almost surely the maximal achievable value, $\log3/\log2$.

\subsubsection*{Example 3}
We now show that limit sets of Mandelbrot percolation are random recursive sets.
Recall that $k$-fold Mandelbrot percolation of the $d$-dimensional unit cube for threshold value $0<p<1$ is defined recursively in the following way:
Let $Q_1$ be the set containing the unit cube. The set $Q'_{k+1}$ is defined as the set of all cubes that are obtained by splitting all cubes in $Q_k$ into $k^d$ smaller cubes of the same dimensions to obtain $(k^d \cdot \# Q_k)$ subcubes with sidelengths $1/k$. For each cube in $Q'_{k+1}$  we then decide independently with probability $p$ to keep the cube. We set $Q_{k+1}$ to be the set of `surviving' cubes.
This process is called Mandelbrot percolation and the random limit set one obtains is $Q_{\infty}=\bigcap_{k=1}^\infty Q_k$. For $p> 1/k^d$, there exists positive probability that the limit set is non-empty. Conditioned on non-extinction, we have $\dim_{qA}Q_{\infty}=\dim_HQ_{\infty}=\log(k^d p)/\log k$ and $\dim_A Q_{\infty}=d$, since $Q_\infty$ is a random recursive set, a fact we now show. 

For simplicity we assume $k=d=2$, \ie we percolate the unit square and subdivide any subsquare into $2\times 2$ squares, keeping each with probability $1/4<p<1$. Let $f_1$ be the homothety that maps the unit square to $[0,1/2]\times[0,1/2]$, let $f_2$ be the homothety that maps $[0,1]\times[0,1]$ to $[0,1/2]\times[1/2,1]$ and similarly let $f_3$ map $[0,1]\times[0,1]$ to $[1/2,1]\times[1/2,1]$ and $f_4$ map $[0,1]\times[0,1]$ to $[1/2,1]\times[0,1/2]$.
We define $\Lambda=\{0,1,\dots,15\}$, $\mathbb{L}=\{\mathbb{I}_0,\mathbb{I}_1,\dots, \mathbb{I}_{15}\}$, and $\mu=\sum_{i=0}^{15} q_i \delta_i$, with $\mathbb{I}_i$ and $q_i$ given in the table below.
\vspace{-0.5cm}
\begin{center}
	\begin{tabular}{ccc}
	\begin{tabular}[t]{|c|c|c|}\hline
		$i$&$q_i$&$\mathbb{I}_i$\\\hline\hline
		$0$ & $(1-p)^4$ & $\varnothing$\\\hline
		$1$ & $p(1-p)^3$ & $\{f_1\}$\\\hline
		$2$ & $p(1-p)^3$ & $\{f_2\}$\\\hline
		$3$ & $p(1-p)^3$ & $\{f_3\}$\\\hline
		$4$ & $p(1-p)^3$ & $\{f_4\}$\\\hline
		$5$ & $p^2(1-p)^2$ & $\{f_1,f_2\}$\\\hline
	\end{tabular}
	\begin{tabular}[t]{|c|c|c|}\hline
		$i$&$q_i$&$\mathbb{I}_i$\\\hline\hline
		$6$ & $p^2(1-p)^2$ & $\{f_1,f_3\}$\\\hline
		$7$ & $p^2(1-p)^2$ & $\{f_1,f_4\}$\\\hline
		$8$ & $p^2(1-p)^2$ & $\{f_2,f_3\}$\\\hline
		$9$ & $p^2(1-p)^2$ & $\{f_2,f_4\}$\\\hline
		$10$ & $p^2(1-p)^2$ & $\{f_3,f_4\}$\\\hline
		$11$ & $p^3(1-p)$ & $\{f_1,f_2,f_3\}$\\\hline
	\end{tabular}
	\begin{tabular}[t]{|c|c|c|}\hline
		$i$&$q_i$&$\mathbb{I}_i$\\\hline\hline
		$12$ & $p^3(1-p)$ & $\{f_1,f_2,f_4\}$\\\hline
		$13$ & $p^3(1-p)$ & $\{f_1,f_3,f_4\}$\\\hline
		$14$ & $p^3(1-p)$ & $\{f_2,f_3,f_4\}$\\\hline
		$15$ & $p^4$ & $\{f_1,f_2,f_3,f_4\}$\\\hline
	\end{tabular}
\end{tabular}
\end{center}
\vspace{0.3cm}

In a similar way fractal percolation in the sense of Falconer and Jin, see~\cite{Falconer15}, is a random recursive construction. Let $\mathbb{I}'$ be an IFS and consider its deterministic attractor $F$. Fix $p>(\#\mathbb{I})^{-1}$. The limit set of fractal percolation is obtained by percolation of the tree associated with $F$, keeping subbranches with probability $p$ and deleting them with probability $(1-p)$.

\subsection*{Acknowledgements}
The author thanks Kathryn Hare and Franklin Mendivil for proposing this question. The author also wishes to thank Jonathan Fraser for pointing out that the main theorem immediately implies the trivial Assouad spectrum. Lastly, the author thanks Hui He for bringing \cite{Athreya94} to our attention. 

\bibliographystyle{../../Biblio/stcustom}
\bibliography{../../Biblio/Biblio}

\end{document}